\documentclass[12pt]{amsart}
\usepackage{amsmath,amssymb,amsthm}

\usepackage{fullpage}

\newtheorem{lem}{Lemma}[section]
\newtheorem{prop}{Proposition}[section]

\newtheorem{thm}{Theorem}[section]

\newtheorem{example}{Example}[section]

\theoremstyle{definition}

\theoremstyle{remark}

\theoremstyle{remark}
\newtheorem{remark}{Remark}[section]
\numberwithin{equation}{section}

\newcommand{\N}{{\mathbb N}}

\newcommand{\R}{{\mathbb R}}

\DeclareMathOperator{\re}{Re}

\begin{document}

\title{Maximizers for Gagliardo--Nirenberg inequalities\\ and related non-local problems}

\begin{abstract}
In this paper we study the existence of maximizers for two families of interpolation inequalities, namely a generalized Gagliardo--Nirenberg inequality and a new inequality involving the Riesz energy. Two basic tools in our argument are a generalization of Lieb's Translation Lemma and a Riesz energy version of the Br\'ezis--Lieb lemma. 
\end{abstract}

\author{Jacopo Bellazzini}
\address{Jacopo Bellazzini\\
Universit\`a di Sassari\\
Via Piandanna 4, 07100 Sassari, Italy}
\email{jbellazzini@uniss.it}

\author{Rupert L. Frank}
\address{Rupert L. Frank\\
Mathematics 253-37, Caltech\\
Pasadena, CA 91125, USA}
\email{rlfrank@caltech.edu}

\author{Nicola Visciglia}
\address{Nicola Visciglia\\
Dipartimento di Matematica Universit\`a di Pisa\\
Largo B. Pontecorvo 5, 56100 Pisa, Italy}
\email{viscigli@dm.unipi.it}

\maketitle

\section{Introduction}

An old and obvious, but useful observation is that boundedness of a linear operator $A$ from some Banach space $B$ to $L^q(X)$, $X$ a measure space, follows, provided one knows boundedness from $B$ to both $L^{q_1}(X)$ and $L^{q_2}(X)$ for some $1\leq q_1<q<q_2\leq\infty$. In this paper we address the question whether the norm of $A$, acting from $B$ to $L^q(X)$, is attained, that is, whether the supremum
\begin{equation}
\label{eq:toy}
\sup_{f\neq 0} \frac{\|Af\|_{L^q}}{\|f\|_B}
\end{equation}
is a maximum. We shall see that in many situations, the \emph{boundedness} from $B$ to both $L^{q_1}(X)$ and $L^{q_2}(X)$ provides the key input in proving that the norm from $B$ to $L^q(X)$ is \emph{attained}.

Instead of developing an abstract theory, we explain our argument with two examples, the first one being the following Gagliardo--Nirenberg inequality,
\begin{equation}\label{GagliNiren}
\|D^r \varphi\|_{L^q}\leq C(r,s,p,q, d) \, \|\varphi\|_{L^p}^{1-\theta} \, \|D^s \varphi\|_{L^2}^{\theta}
\qquad \forall \varphi\in \dot H^s(\R^d)\cap L^p(\R^d) \,,
\end{equation}
where
$$
D=\sqrt{-\Delta} \,.
$$
It is known that \eqref{GagliNiren} holds provided that
\begin{enumerate}
\item $-r+\frac{d}{q}=(1-\theta)\frac{d}{p}+\theta(-s+\frac{d}{2})$,
\item $\frac{r}{s}\leq \theta\leq 1$,
\item $0<r\leq s$, $1<p,q<\infty$.
\end{enumerate}
Of course, condition (1) expresses the scale invariance of \eqref{GagliNiren}. The proof of \eqref{GagliNiren} is essentially contained in \cite{GiVe} (see our Lemma \ref{fracGN} below) and stated explicitly, for instance, in \cite{NP}. 

Note that finding the smallest possible constant $C(r,s,p,q,d)$ in \eqref{GagliNiren} is a problem of the form \eqref{eq:toy} with $B=\dot H^{s}(\R^d)\cap L^p(\R^d)$, $X=\R^d$ and $Af = D^r f$. Indeed, finding the best constant in \eqref{GagliNiren} is, by scaling, equivalent to finding the best constant in the inequality with the right side replaced by the norm $\|f\|_B=
\left( \|f\|_{L^p}^2 + \|D^s f\|_{L^2}^2\right)^{1/2}$. Besides the best constants also the optimizers (if they exist) are related by scaling.

In Theorem \ref{main} we shall show that the supremum
$$
\sup_{\varphi\neq 0} \frac{\|D^r \varphi\|_{L^q}}{ \|\varphi\|_{L^p}^{1-\theta}\|D^s \varphi\|_{L^2}^{\theta}}
$$
is attained under the additional assumption $r/s<\theta<1$. (Compare with (2)!) This is precisely the condition that one can find $q_1<q<q_2$ for which the inequality is valid (with $r,s,p$ and $d$ held fixed). In Example \ref{ex} we shall show that, in general, the supremum may not be attained if $\theta=r/s$.

The result in this generality seems to be new. The important special case $r=0, s=1$, $p=2$ is a classical result of Weinstein \cite{W}. We also mention the work \cite{DD} where the optimal constant and the optimizers were explicitly found in the special case $r=0$, $s=1$, $q=\frac p2 +1$, and \cite{CE} for a mass transportation approach to the problem.

\bigskip

Our second motivation, and our second example illustrating the method described above, concerns a lower bound for the quantity
\begin{equation}
\label{eq:d}
\iint_{\R^d\times\R^d} \frac{|\varphi(x)|^2 |\varphi(y)|^2}{|x-y|^\lambda} \,dxdy \,.
\end{equation}
In the special case $\lambda=d-2$, $d\geq 3$, this expression represents the Coulomb energy of a density $|\varphi|^2$. While upper bounds on \eqref{eq:d} are readily available (for instance in terms of an $L^p(\R^d)$ norm of $\varphi$ by the Hardy--Littlewood--Sobolev inequality; see, e.g., \cite{LiLo}), lower bounds are a notoriously difficult problem in mathematical physics. Among the few available results are the Lieb--Oxford inequality \cite{LiOx}. In Proposition \ref{CHLSgen} we prove a lower bound for \eqref{eq:d}, a special case of which is, for $d=3$ and $\lambda=1$,
\begin{equation}
\label{eq:dbound}
\|\varphi\|_{L^{2p}}\leq C(p,s) \|\varphi\|_{\dot H^{s}}^{\frac{\theta}{2-\theta}}\left(\iint_{\R^3\times \R^3}
\frac{|\varphi(x)|^2 |\varphi(y)|^2}{|x-y|} \, dxdy\right)^{\frac{1-\theta}{4-2\theta}}
\end{equation}
with $\theta=\frac{6-5p}{3-2ps-2p}$. Here the parameters $s>0$ and $1<p\leq\infty$ satisfy
\begin{align*}
& p\in \left[\frac{3}{3-2s}, \frac{1+2s}{1 +s}\right] & & \text{if}\ 0<s<1/4 \,, \\
& p = \frac{3}{3-2s}= \frac{1+2s}{1 +s} & & \text{if}\ s=1/4 \,, \\
& p\in\left[ \frac{1+2s}{1 +s},  \frac 3{3-2s}\right] & & \text{if}\ 1/4< s<3/2 \,, \\
& p\in \left[\frac{1+2s}{1 +s}, \infty\right) & & \text{if}\ s= 3/2 \,, \\
& p\in \left[\frac{1+2s}{1+s}, \infty\right] & & \text{if}\ s> 3/2 \,.
\end{align*}
In Theorem \ref{smag1} we show, among other things, that the supremum
$$
\sup_{\varphi\not\equiv 0} \frac{\|\varphi\|_{L^{2p}}}{\|\varphi\|_{\dot H^{s}}^{\frac{\theta}{2-\theta}}\left(\iint_{\R^3\times \R^3}
\frac{|\varphi(x)|^2 |\varphi(y)|^2}{|x-y|}dxdy\right)^{\frac{1-\theta}{4-2\theta}}}
$$
is attained under the additional assumption $p\neq 3/(3-2s)$, $p\neq (1+2s)/(1+s)$ and $p\neq\infty$. This assumption has the same origin as explained before.

\bigskip

Our third goal is to advertize a compactness proof which has its origins in the works \cite{L,Lieb,BL,FLL}. One ingredient in this compactness proof is the \emph{pqr} Lemma of Fr\"ohlich, Lieb and Loss \cite{FLL} which lies at the heart of our interpolation approach. The two other ingredients are a `compactness up to translation' lemma \cite{Lieb} and the Br\'ezis--Lieb lemma about the remainder term in Fatou's lemma \cite{BL}. When carrying out the optimization in our two examples we have to prove both a new version of the `compactness up to translations' lemma (Lemma \ref{LiebIntro}), which is valid in $\dot H^s(\R^d)$ with arbitrary $s>0$, and a non-local version of the Br\'ezis--Lieb lemma (Theorem \ref{BL}), which is valid for the double integral \eqref{eq:d}. We hope that both results are of interest even beyond the concrete context of this paper.

Of course, we are aware that there are alternative approaches to establish the existence of optimizers in variational problems. Early approaches are based on symmetrization, but are therefore restricted to $\dot H^s$ with $0<s\leq 1$; see, e.g., \cite{T,St,L0,W,L} and references therein. Lions developed his method of concentration compactness in \cite{Li1,Li2} and G\'erard found a different argument based on refined Sobolev inequalities \cite{Ge,GMO}; see also \cite{FL} for a recent application of this technique and also \cite{PP}. In addition, progress was made concerning the existence of optimizers in restriction-type inequalities; see, e.g., \cite{BR}, \cite{CS}, \cite{D}, \cite{FVV}, \cite{FVV2}, \cite{Fo}, \cite{Q}.


\section{Main results}

Our first main result is

\begin{thm}[Existence of optimizers for the Gagliardo--Nirenberg inequality] \label{main}
Let $r,s>0$, $\frac{r}{s}<\theta<1$ and $p, q\in (1, \infty)$ such that
$$
-r+\frac{d}{q}=(1-\theta)\frac{d}{p}+\theta\left(-s+\frac{d}{2}\right) \,.
$$
Then
\begin{equation}
\label{eq:main}
\sup_{0\not\equiv\varphi\in \dot H^s\cap L^p} \frac{\|D^r \varphi\|_{L^q}}{ \|\varphi\|_{L^p}^{1-\theta}\|D^s \varphi\|_{L^2}^{\theta}}
\end{equation}
is attained.
\end{thm}

Notice that we don't assume on $s$ any extra restriction of the type $s<d/2$, and hence no Sobolev 
embedding in principle is available.
We underline that this is not a purely mathematical question, in fact
the corresponding Gagliardo--Nirenberg inequalities (with $s>d/2$) play a crucial role
in several models (see \cite{FIP}, \cite{PS}).

As explained in the introduction, the crucial assumption is that $\frac{r}{s}<\theta<1$. The following example shows that in the endpoint case $\theta=r/s$ we cannot, in general, expect the existence of maximizers for \eqref{GagliNiren}.

\begin{example}\label{ex}
Consider the Gagliardo--Nirenberg inequality (for simplicity we work in dimension $d=1$, but the example can be extended in higher dimensions)
\begin{equation}\label{GNnomax}
\|u'\|_{L^2}\leq \|u\|_{L^2}^{1/2}\|u''\|_{L^2}^{1/2} \,.
\end{equation}
By Plancharel's theorem this inequality is equivalent to the inequality
\begin{equation*}
\|\xi \hat u\|_{L^2}\leq \|\hat u\|_{L^2}^{1/2}\| \xi^2 \hat u\|_{L^2}^{1/2} \,,
\end{equation*}
which is an immediate consequence of the Cauchy--Schwarz inequality. Moreover, by choosing $$\hat u_n (\xi)=\chi_{(1-1/n, 1+1/n)}(\xi)$$ one can check that the constant $C=1$ on the right side cannot be improved. In fact, it is easy to show that there are not optimizers in the inequality above since the Cauchy--Schwarz inequality gives an identity if and only if the functions
involved are multiples of each other. Hence any maximizer $\hat u_0$ should satisfy
$$\xi^2 |\hat u_0(\xi)|=\mu |\hat u_0(\xi)|$$
for a suitable $\mu\in \R$, which is only possible for $\hat u_0=0$.
\end{example}

In the other endpoint case $\theta=1$ inequality \eqref{GagliNiren} turns into a Sobolev inequality and compactness may be lost both by translations and dilations. For $\theta=1$, $r=0$ and $s<d/2$ existence of optimizers was proved in \cite{L} (in a dual formulation).

\begin{remark}
If $r\in\N$, then a slight modification of our proof shows that the supremum is also attained if $\| D^r \varphi\|_{L^q}$ in \eqref{eq:main} is replaced by
\begin{equation}
\label{eq:normint}
\left( \sum_{\substack{\alpha=(\alpha_1,...,\alpha_d) \\
\sum_{i=1}^d \alpha_i=r}} \int_{\R^d} |\partial^{\alpha} \varphi|^q dx \right)^{1/q} \,.
\end{equation}
A similar remark applies to $\| D^s \varphi\|_{L^2}$. If both $r$ and $s$ are integers, \eqref{eq:main} with norms as in \eqref{eq:normint} is proved, for instance, in \cite{Ca} and \cite{Ma}.
\end{remark}

We now discuss our second main result, which gives a lower bound on \eqref{eq:d} in terms of an $L^p$ norm and a homogeneous Sobolev norm of $\varphi$. We define the Banach spaces ${\mathcal H^{s,\lambda}}(\R^d)$ as the completion of $C^\infty_0 (\R ^d)$ with respect to the norm
$$
\|\varphi\|_{\mathcal H^{s, \lambda}}:=\|\varphi\|_{\dot H^{s}} + \left( \iint_{\R^d\times \R^d}
\frac{|\varphi(x)|^2 |\varphi(y)|^2}{|x-y|^{\lambda}}dxdy\right )^{\frac 14} \,.
$$
Arguing as in \cite{R} one can prove that this quantity is a norm.

\begin{prop}\label{CHLSgen}
Let $d\geq 1$, $s>0$, $0<\lambda<d$ and
assume that $p$ satisfies
\begin{enumerate}
\item if $\lambda< 4s$ and $ d< 2s$ then
$$p\in \left[\frac{d-\lambda+4s}{d-\lambda +2s}, \infty\right] \,,$$
\item if $\lambda< 4s$ and $ d= 2s$ then
$$p\in \left[\frac{d-\lambda+4s}{d-\lambda +2s}, \infty\right) \,,$$
\item if $\lambda< 4s$ and $d>2s$ then
$$p\in\left[ \frac{d-\lambda+4s}{d-\lambda +2s},  \frac d{d-2s}\right] \,,$$
\item if $\lambda= 4s$ and $d>2s$ then
$$p= \frac{d-\lambda+4s}{d-\lambda +2s}=  \frac d{d-2s} \,,$$
\item if $\lambda>4s$ and $d>2s$
then
$$p\in \left[\frac{d}{d-2s}, \frac{d-\lambda+4s}{d-\lambda +2s}\right] \,.$$
\end{enumerate}
Then the  following bound holds,
\begin{equation}\label{converseHLSgen}
\|\varphi\|_{L^{2p}}\leq C(p,d,s,\lambda) \|\varphi\|_{\dot H^{s}}^{\frac{\theta}{2-\theta}}\left(\iint_{\R^d\times \R^d}
\frac{|\varphi(x)|^2 |\varphi(y)|^2}{|x-y|^{\lambda}}dxdy\right)^{\frac{1-\theta}{4-2\theta}} 
\qquad \forall \varphi \in \mathcal H^{s,\lambda}(\R^d)
\end{equation}
with $\theta=\frac{2d-2pd+p\lambda}{d-2ps-pd+p\lambda}$.
\end{prop}

The inequality in the proposition for $\lambda=1, s=1,d=3$ is well-known. In this case the relevant range is $2p\in [3, 6]$. Remarkably, it is proved in \cite{R} that this range of $p$ can be increased if the inequality is restricted to radial functions. In the case $\lambda=1$, $s=\frac 12$, $d=3$ the inequality of the proposition was first obtained in \cite{BOV}.

Our second main result states the existence of maximizers for the inequalities
\eqref{converseHLSgen} in the non-endpoint case.

\begin{thm}[Existence of optimizers for \eqref{converseHLSgen}]\label{smag1}
Let  $s, \lambda, d,p$ be as in Theorem \ref{CHLSgen}. Assume, moreover, that
\begin{equation}\label{noendpoint}
p\neq \frac{d-\lambda+4s}{d-\lambda +2s} \,,
\qquad
p\neq \frac{d}{d-2s} \,,
\qquad
p\neq \infty \,.
\end{equation}
Then the best constant in \eqref{converseHLSgen} is achieved.
\end{thm}

We conclude this section by presenting two technical tools which play a crucial role in our compactness proof. To motivate the first one we note that Theorems \ref{main} and \ref{smag1} follow easily provided one is able to show that the weak limit of a maximizing sequence is, up to translation, different from zero. This is exactly the content of the following Lemma \ref{LiebIntro}, which is a generalization of Lieb's compactness lemma \cite{Lieb}.

\begin{lem}[Compactness up to translations in $\dot H^{s}$]\label{LiebIntro}
Let $s>0$, $1<p<\infty$ and $u_n\in \dot H^s(\R^d)\cap L^p(\R^d)$ be a sequence with
\begin{equation}\label{hyp1}
\sup_n \left( \|u_n\|_{\dot H^{s}} +\|u_n\|_{L^p} \right) <\infty
\end{equation}
and, for some $\eta>0$, (with $|\cdot |$ denoting Lebesgue measure)
\begin{equation}\label{hyp2}
\inf_n \left|\{ |u_n|>\eta \}\right|>0 \,.
\end{equation}
Then there is a sequence $(x_n)\subset\R^d$ such that a subsequence of $u_n(\cdot+ x_n)$ has a weak limit $u\not\equiv 0$ in $\dot H^s(\R^d)\cap L^p(\R^d)$.
\end{lem}

Lemma \ref{LiebIntro} for $s=1$, $p=2$ is due to Lieb \cite{Lieb}.

Note that as in theorem \ref{main} we do \emph{not} assume that $s<d/2$. If this is satisfied, however, then $\sup_n \|u_n\|_{\dot H^s}<\infty$ implies $\sup_n \|u_n\|_{L^p}<\infty$ for $p=2d/(d-2s)$ by a Sobolev inequality and therefore the statement of the lemma can be simplified. In particular, for $s=1$ we recover the version of Lieb's compactness lemma as proved in \cite{BB}.

Lieb's proof of the compactness lemma for $s=1$ depends heavily on the pointwise chain rule for the gradient. 
Since this is not available for general $s$ we develop a different approach. Instead of space localization, which is more complicated to handle in the case of fractional Sobolev spaces, we shall exploit momentum localization. We will also make use of a refined Sobolev inequality in Besov spaces, see Theorem \ref{gmo}. Since we do not assume that $s<d/2$ we need an additional ingredient, in addition to the refined Sobolev inequality. More precisely, we treat separately the low- and high-frequency parts of the sequence.

Our second technical ingredient is the following `non-local version' of the Br\'ezis--Lieb lemma.
\begin{lem}\label{BL}
Let $0<\lambda<d$, let $2<p<\infty$ and let $f_n$ and $f$ be functions on $\R^d$ such that
\begin{equation}
\label{punt}
f_n(x)\rightarrow f(x) \hbox{ a.e. } x\in \R^d \,,
\end{equation}
\begin{equation}\label{nonlocal}
\sup_n \iint_{\R^d\times\R^d} \frac{|f_n(x)|^2|f_n(y)|^2}{|x-y|^\lambda}dxdy <\infty \,,
\end{equation}
\begin{equation}\label{sobolev}
\sup_n \|f_n\|_{L^p}<\infty \,.
\end{equation}
Then
\begin{align}
\iint_{\R^d\times\R^d} \frac{|f_n(x)|^2|f_n(y)|^2}{|x-y|^\lambda}dxdy 
= & \iint_{\R^d\times\R^d} \frac{|f(x)|^2|f(y)|^2}{|x-y|^\lambda}dxdy \notag \\
& + \iint_{\R^d\times\R^d} \frac{|f_n(x)-f(x)|^2 |f_n(y)-f(y)|^2}{|x-y|^\lambda} dxdy + o(1) \label{con} \,.
\end{align}
\end{lem}

This lemma is due to \cite{MS} in the special case $p=2d/(2d-\lambda)$. In this special case one can use the Hardy--Littlewood--Sobolev inequality to control the double integral in \eqref{nonlocal} by the norm in \eqref{sobolev}. We do not see how the argument in \cite{MS} can be generalized to arbitrary $2<p<\infty$ and we provide an independent proof.

For easier reference we state

\begin{lem}[\emph{pqr} Lemma \cite{FLL}]\label{Lieb}
Let $1\leq p<q<r\leq\infty$ and let $\alpha, \beta, \gamma>0$. Then there are constants $\eta,c>0$ such that for any measurable function $f\in L^p(X)\cap L^r(X)$, $X$ a measure space, with
\begin{equation*}
\|f\|_{L^p}^p\leq \alpha, \quad
\|f\|_{L^q}^q\geq \beta, \quad
\|f\|_{L^r}^r\leq \gamma, \quad
\end{equation*}
one has (with $|\cdot|$ denoting the underlying measure on $X$)
\begin{equation}\label{statement}
\left|\{ x \in X :\ |f(x)|>\eta \}\right| \geq c \,.
\end{equation}
\end{lem}

The paper is organized as follows: Sections \ref{main2} and \ref{section3} are devoted to the proof of Theorems~\ref{main} and \ref{smag1}, respectively. Sections \ref{section4} and \ref{sec:bl} contain the proof of Lemmas \ref{LiebIntro} and \ref{BL}. In the appendix we derive Proposition \ref{CHLSgen}.  


\section{Existence of optimizers for the Gagliardo--Nirenberg inequality}\label{main2}

In the proof of Theorem \ref{main} we assume the fact that $C(r,s,p,q,d)<\infty$. As we mentioned in the introduction, this is essentially contained in \cite{GiVe} (and can be deduced by the same method as in our Lemma \ref{fracGN} below) and is explicitly stated, for instance, in \cite{NP}.

Since inequality \eqref{GagliNiren} is invariant under 
homogeneity $\varphi(x) \mapsto \lambda \varphi(x)$ and scaling 
$\varphi \mapsto \varphi(\lambda x)$ for any $\lambda>0$, we can choose a maximizing sequence $\varphi_n$ such that 
\begin{equation}\label{0.0}
\|D^r \varphi_n\|_{L^q}= C(r,s,p,q,d)+o(1)
\end{equation}
and
\begin{equation}\label{0.1}
\|\varphi_n\|_{L^p}=\|D^s \varphi_n\|_{L^2}=1 \,.
\end{equation}

The key observation is that, since we are looking at a non-endpoint case (i.e., $r/s<\theta<1$), there are $q_1<q<q_2$ such that the scaling relations
$$
-r + \frac{d}{q_j}=(1-\theta_j)\frac{d}{p} + \theta_j\left(-s+\frac{d}{2}\right)
$$
are satisfied for some $\theta_j$ satisfying $r/s<\theta_j<1$. From inequality \eqref{GagliNiren} with these values of $q_j$ (and the same $r$, $s$ and $p$) we infer that
\begin{equation}\label{inequ}
\sup_n \max\left\{\|D^r \varphi_n\|_{L^{q_1}}, 
\|D^r \varphi_n\|_{L^{q_2}}\right\}<\infty \,.
\end{equation} 
The $pqr$-lemma (Lemma \ref{Lieb}) now implies that
\begin{equation}\label{superlevel}
\inf_n \left|\{ |D^r \varphi_n|>\eta \}\right| >0
\end{equation}
for some $\eta>0$.

Next, we apply the compactness modulo translations lemma (Lemma \ref{LiebIntro}) to the sequence $(D^r \varphi_n)$. This sequence is bounded in $\dot H^{s-r}$ by \eqref{0.1}, and \eqref{hyp1} and \eqref{hyp2} are satisfied by \eqref{0.0} and \eqref{superlevel}. Thus, there is a sequence $(x_n)\subset\R^d$ such that, after passing to a subsequence if necessary, we have $D^r \varphi_n (\cdot+x_n) \rightharpoonup \psi\not\equiv 0$ in $\dot H^{s-r}\cap L^q$.

The translated sequence $\tilde\varphi_n(x)=\varphi_n(x+x_n)$ still satisfies \eqref{0.0} and \eqref{0.1} and therefore has a subsequence which converges weakly in $\dot H^{s}\cap L^p$ to some $\tilde\varphi$. For $\chi\in C_0^\infty(\R^d)$ we have that
$$
\int_{\R^d} \tilde\varphi D^r \chi \,dx 
= \lim_{n\to\infty} \int_{\R^d} \varphi(x+x_n) D^r \chi \,dx
= \lim_{n\to\infty} \int_{\R^d} D^r \varphi(x+x_n) \chi \,dx
= \int_{\R^d} \psi \chi \,dx \,,
$$
and therefore $D^r \tilde\varphi = \psi$. In particular, $\tilde\varphi\not\equiv 0$.

By Rellich's compactness theorem we can pass to a further subsequence to ensure that
\begin{equation}
\label{eq:aeconv}
D^r \tilde\varphi_n \to D^r \tilde\varphi \quad \text{a.e.}
\qquad\text{and}\qquad
\tilde\varphi_n \to \tilde\varphi \quad \text{a.e.}
\end{equation}

It remains to prove that $\tilde\varphi$ is an optimizer. In order to simplify notation we drop the tildes (that is, we assume that the $\varphi_n$ are chosen at the beginning so that $x_n=0$). By \eqref{0.1} and the Hilbert structure of $\dot H^s$ we get
\begin{equation}\label{eq3}
\|D^s \varphi_n-D^s \varphi\|_{L^2}^2+\|D^s \varphi\|_{L^2}^2=1+o(1) \,,
\end{equation}
and by \eqref{0.0}, \eqref{0.1}, \eqref{eq:aeconv} and the Br\'ezis--Lieb lemma \cite{BL} we get
\begin{equation}\label{eq1}
\|D^r \varphi_n-D^r \varphi\|_{L^q}^q+\| D^r \varphi\|_{L^q}^q=C^q+o(1)
\end{equation}
and
\begin{equation}\label{eq2}
\|\varphi_n-\varphi\|_{L^p}^p+\| \varphi\|_{L^p}^p=1+o(1) \,,
\end{equation}
where we abbreviated $C=C(r,s,p,q,d)$. By combining \eqref{eq1}, \eqref{eq2}, \eqref{eq3} we get
\begin{align}
\label{bycont}
& C^{-q} \left( \| D^r \varphi\|_{L^q}^q + \|D^r \varphi_n-D^r \varphi\|_{L^q}^q+o(1) \right) \\
& = \left( \| \varphi\|_{L^p}^p + \|\varphi_n-\varphi\|_{L^p}^p+ o(1) \right)^{\frac{q(1-\theta)}p} \left( \|D^s \varphi\|_{L^2}^2+ \|D^s \varphi_n-D^s \varphi\|_{L^2}^2+o(1) \right)^{\frac{\theta q}2} \,. \notag
\end{align}
In view of the elementary inequality
$$
(1+x)^\alpha (1+y)^\beta \geq 1 + x^\alpha y^\beta
$$
for all $x,y\geq 0$ if $\alpha+\beta\geq 1$, we have
$$
\left( a^p + b^p \right)^{\frac{(1-\theta)q}{p}} \left(c^2+d^2 \right)^{\frac{\theta q}{2}} \geq a^{(1-\theta)q} c^{\theta q} + b^{(1-\theta)q} d^{\theta q}
$$
for all $a,b,c,d\geq 0$. Here we used the fact that, by the scaling relation
$$
\frac{(1-\theta)q}{p} + \frac{\theta q}{2} = 1 + \frac{qs}{d}\left(\theta -\frac{r}{s} \right) \geq 1 \,.
$$
Therefore, we can bound the right side of \eqref{bycont} from below by
$$
\| \varphi\|_{L^p}^{(1-\theta)q} \| D^s\varphi \|_{L^2}^{\theta q} + \| \varphi_n-\varphi\|_{L^p}^{(1-\theta)q} \| D^s\varphi_n - D^s\varphi \|_{L^2}^{\theta q}  + o(1) \,,
$$
which in turn can be bounded from below via \eqref{GagliNiren} by
$$
\| \varphi\|_{L^p}^{(1-\theta)q} \| D^s\varphi \|_{L^2}^{\theta q} + 
C^{-q} \|D^r \varphi_n-D^r \varphi\|_{L^q}^q + o(1) \,.
$$
We insert this lower bound in \eqref{bycont} and conclude that
$$
C^{-q} \| D^r \varphi\|_{L^q}^q +o(1) \geq \| \varphi\|_{L^p}^{(1-\theta)q} \| D^s\varphi \|_{L^2}^{\theta q} + o(1) \,,
$$
which means that $\varphi$ is an optimizer. This completes the proof.
\qed


\section{Proof of Theorem \ref{smag1}}\label{section3}

The argument is similar to that used in the proof of Theorem \ref{main}. Again by homogeneity and scaling we can assume that an optimizing sequence $\varphi_n\in\mathcal H^{s,\lambda}(\R^d)$ satisfies
$$
\|\varphi_n \|_{\dot H^s}=
\iint_{\R^d\times\R^d} \frac{|\varphi_n(x)|^2\ |\varphi_n(y)|^2}{|x-y|^\lambda} \,dx\,dy = 1
$$
and
$$
\| \varphi_n \|_{L^{2p}} = C(p,d,s,\lambda) + o(1) \,.
$$
Since we are not at the endpoint, we can argue as in the proof of Theorem \ref{main} and find a uniform upper bound on $\|\varphi_n\|_{L^{2p_1}}$ and $\|\varphi_n\|_{L^{2p_2}}$ for some $p_1<p<p_2$. Therefore, by Lemma~\ref{Lieb}
$$
\inf_n \left|\{ |\varphi_n|>\eta \}\right| >0 \,.
$$
Thus, after a translation if necessary, we may assume that there is a $\varphi\not\equiv 0$ such that $\varphi_n\rightharpoonup\varphi\in\dot H^s(\R^d)\cap L^p(\R^d)$. Thus, by the Hilbert structure of $\dot H^s(\R^d)$,
\begin{equation}
\label{eq:decomp2t}
\|\varphi_n - \varphi\|_{\dot H^s}^2 + \|\varphi\|_{\dot H^s}^2 = 1+ o(1) \,.
\end{equation}

By Rellich's compactness theorem we may assume, after passing to another subsequence if necessary, that $\varphi_n\to\varphi$ almost everywhere, and therefore by the Brezis--Lieb lemma \cite{BL},
\begin{equation}
\label{eq:decomp2m}
\|\varphi_n - \varphi\|_{L^{2p}}^{2p} + \|\varphi\|_{L^{2p}}^{2p} = C^{2p} + o(1) \,,
\end{equation}
where we abbreviated $C=C(p,d,s,\lambda)$.

Finally, by our non-local Brezis--Lieb lemma (Theorem \ref{BL}),
\begin{equation}
\label{eq:decomp2u}
\iint_{\R^d\times\R^d} \frac{|\varphi_n(x)-\varphi(x)|^2\ |\varphi_n(y)-\varphi(y)|^2}{|x-y|^\lambda} \,dx\,dy +
\iint_{\R^d\times\R^d} \frac{|\varphi(x)|^2\ |\varphi(y)|^2}{|x-y|^\lambda} \,dx\,dy = 1 + o(1) \,.
\end{equation}

With \eqref{eq:decomp2t}, \eqref{eq:decomp2m} and \eqref{eq:decomp2u} at hand, the proof of Theorem \ref{smag1} follows along the same lines as that of Theorem \ref{main}.


\section{Compactness up to translations} \label{section4}

Our goal in this section is to prove Lemma \ref{LiebIntro}. We begin by recalling the refined Sobolev inequality of Gerard--Meyer--Oru, see \cite{GMO}, in the form given in \cite{BCD}.

\begin{thm}\label{gmo}
Let $0<s<d/2$ and let $\theta\in\mathcal S(\R^d)$ be such that $\hat\theta$ has compact support, has value $1$ near the origin and satisfies $0\leq\hat\theta\leq 1$. Then there is a constant $C=C_{s,d}(\theta)$ such that for all $u\in \dot H^s(\R^d)$,
$$
\|u\|_q \leq C \| u\|_{\dot H^s}^{2/q} \left( \sup_{A>0} A^{d/2+s} \|\theta(A\,\cdot) \star u\|_\infty \right)^{1-2/q}
$$
with $q=2d/(d-2s)$.
\end{thm}

One can show that the supremum on the right side is equivalent to the norm in a certain Besov space, but we will not need this fact.

Several applications of this improvement have been given in the literature.
In particular we quote the profile decomposition theorems in \cite{Ge}, \cite{Ke}
and its generalization in \cite{FV}. The main point is that the 
use of the aforementioned improved Sobolev embedding involves in principle
the introduction of suitable dilation and translation parameters. However we notice a-posteriori that the
introduction of the dilation parameter can be ignored, since it stays bounded away from zero and infinity. Hence we get at the end 
a proof of Lemma
\ref{LiebIntro} where only translation parameters are involved.  

We are now in position to give the

\begin{proof}[Proof of Lemma \ref{LiebIntro}]
Let $\chi_0$ and $\chi_1$ be smooth, real-valued functions on $[0,\infty)$ such that $\chi_0 +\chi_1 \equiv 1$ and such that $\chi_0$ has compact support and value $1$ near the origin. We decompose $u_n = \chi_0(|D|)u_n + \chi_1(|D|)u_n$ and, because of \eqref{hyp2}, we may assume that either 
\begin{equation}\label{caseA}
\inf_n \left|\left\{ |\chi_0(|D|) u_n|>\frac{\eta}{2}\right\}\right|>0   \ \ \ \text{ (case A)}
\end{equation}
or
\begin{equation}\label{caseB}                   
\inf_n \left|\left\{ |\chi_1(|D|) u_n|>\frac{\eta}{2}\right\}\right|>0  \ \ \ \text{ (case B)} \,.
\end{equation}
We discuss these two cases separately.

\emph{Case A. Low frequency localization.}
The operator $\chi_0(|D|)$ is a convolution operator with respect to the Schwartz function
$$
h(x)= \int_{\R^d} \chi_0(|\xi|) e^{-i\xi\cdot x} \frac{d\xi}{(2\pi)^d} \,.
$$
Therefore, the classical Young inequality implies that
$$
\|\chi_0(|D|) \varphi\|_{L^\infty}\leq \|h\|_{L^{p'}} \|\varphi\|_{L^p}
\quad\text{and}\quad
\|\nabla \chi_0(|D|) \varphi\|_{L^\infty}\leq \|\nabla h\|_{L^{p'}} \|\varphi\|_{L^p}
$$
for any $1\leq p \leq \infty$. In particular, the functions $\chi_0(|D|)u_n$ are continuous and, because of \eqref{caseA}, there are $x_n\in\R^d$ such that $|(\chi_0(|D|)u_n)(x_n)|>\eta/2$.

The translated sequence $v_n(x)=u_n(x+x_n)$ satisfies again \eqref{hyp1}, and therefore has a subsequence which converges weakly in $\dot H^s \cap L^p$ to some $v$. Since $h\in L^{p'}$, we have
$$
(\chi_0(|D|)u_n)(x_n) = \int_{\R^d} h(x_n-y) u_n(y) \,dy = \int_{\R^d} h(-x) v_n(x) \,dx \to \int_{\R^d} h(-x) v(x) \,dx 
$$
as $n\to\infty$ along the chosen subsequence. Since the absolute value of the left side is bigger than $\eta/2$, we conclude that $v\not\equiv 0$.

\emph{Case B. High frequency localization.}
Pick $\sigma< s$ with $\sigma<d/2$ and let $q=2d/(d-2\sigma)$. Then, since $\chi_1$ is supported away from the origin, \eqref{hyp1} implies that $\sup_n \|\chi_1(|D|)u_n\|_{\dot H^\sigma} <\infty$. Moreover, \eqref{caseB} implies that $\inf_n \|\chi_1(|D|)u_n\|_q>0$. Thus, by Theorem \ref{gmo}, there are $A_n>0$ and $x_n\in\R^d$ such that
$$
\inf_n A_n^{d/2+\sigma} \left|\left(\theta(A_n \,\cdot) \star \chi_1(|D|)u_n\right)(x_n)\right| >0 \,.
$$
Since $\hat\theta$ has compact support and $\chi_1$ is supported away from zero, the operator $u\mapsto \theta(A \,\cdot) \star \chi_1(|D|)u$ is zero for all sufficiently small $A$. Thus, $\inf_n A_n >0$.

The translated and dilated sequence $w_n(x) = A_n^{-d/2+\sigma} \left( \chi_1(|D|)u_n \right)(A_n^{-1}x+x_n)$ satisfies
$$
\sup_n \|w_n\|_{\dot H^\sigma} = \sup_n \|\chi_1(|D|)u_n\|_{\dot H^\sigma}<\infty \,,
$$
and therefore has a subsequence which converges weakly in $\dot H^\sigma\cap L^q$ to some $w$. Since $\theta\in L^{q'}$, we have
\begin{align*}
A_n^{d/2+\sigma} \left(\theta(A_n \,\cdot) \star \chi_1(|D|)u_n\right)(x_n) & = A_n^{d/2+\sigma} \int_{\R^d} \theta(A_n(x_n-y)) \left( \chi_1(|D|)u_n\right)(y) \,dy \\
& = \int_{\R^d} \theta(-x) w_n(x) \,dx \\
& \to \int_{\R^d} \theta(-x)w(x) \,dx
\end{align*}
as $n\to\infty$ along the chosen subsequence. Since the absolute value of the left side remains bounded away from zero, we conclude that $w\not\equiv 0$.

We claim that the $\limsup$ of $A_n$ along the chosen subsequence is not infinite. Indeed, by our choice of $\sigma$ there is a $\tilde\sigma$ with $\sigma<\tilde\sigma<\min\{s,d/2\}$. Then again $\sup_n \|\chi_1(|D|)u_n\|_{\dot H^{\tilde\sigma}} <\infty$, and thus Sobolev's inequality implies that $\sup_n\|\chi_1(|D|)u_n\|_{L^{\tilde q}}<\infty$ for $\tilde q=2d/(d-2\tilde\sigma)$. By scaling we have
$$
\|\chi_1(|D|)u_n\|_{L^{\tilde q}}^{\tilde q} = A_n^{d(\tilde q/q -1)} \| w_n\|_{L^{\tilde q}}^{\tilde q} \,.
$$
An elementary fact about weak convergence (see Lemma \ref{weakconv} below) implies that (along the chosen subsequence)
$$
\liminf_{n\to\infty} \|w_n\|_{L^{\tilde q}} \geq \|w\|_{L^{\tilde q}} \,,
$$
where the right side might (a priori) be $+\infty$, but is never zero. But the $L^{\tilde q}$-boundedness now implies that
$$
\infty> \sup\|\chi_1(|D|)u_n\|_{L^{\tilde q}}^{\tilde q} \geq \limsup_{n\to\infty} A_n^{d(\tilde q/q -1)} \|w\|_{L^{\tilde q}}^{\tilde q} \,.
$$
Since $\tilde q>q$ and $\inf_n A_n>0$ we conclude that $w\in L^{\tilde q}$ and $\limsup_n A_n <\infty$. Thus, up to extracting another subsequence, we may assume that $A_n\to A\in (0,\infty)$.

Now let $v_n(x)=u_n(x+x_n)$. Since this sequence satisfies again \eqref{hyp1}, it has a subsequence which converges weakly in $\dot H^s\cap L^p$ to some $v$. Since $\chi_1(|D|)$ is bounded on $\dot H^s\cap L^p$, $\chi_1(|D|)v_n$ converges weakly to $\chi_1(|D|)v$ in $\dot H^s\cap L^p$. On the other hand, $\chi_1(|D|)v_n$ converges weakly in $\dot H^\sigma\cap L^q$ to $A^{d/2-\sigma} w(A\cdot\,)\not\equiv 0$. Thus, $v\not\equiv 0$, as we intended to prove.
\end{proof}

In the proof above we used the following elementary lemma.

\begin{lem}\label{weakconv}
Let $X$ be a measure space and assume that $f_n\rightharpoonup f$ in $L^p(X)$ for some $1< p<\infty$. Then
$$
\int_X |f|^q \,dx \leq \liminf_{n\to\infty} \int_X |f_n|^q \,dx 
$$
for any $1< q<\infty$, in the sense that $\int_X |f_n|^q \,dx\to\infty$ if $f\not\in L^q(X)$.
\end{lem}

\begin{proof}
We may assume that $M=\liminf_n \|f_n\|_q^q<\infty$, for otherwise there is nothing to show. Under this assumption, by weak compactness a subsequence of $f_n$ converges weakly in $L^q$ to some $g$. This implies $\|g\|_q^q \leq M$. Since $L^{p'}\cap L^{q'}$ is dense in both $L^{p'}$ and $L^{q'}$, we have $f=g$, which completes the proof.
\end{proof}


\section{A non-local Br\'ezis--Lieb lemma}\label{sec:bl}

For the proof of Lemma \ref{BL} we need two lemmas.

\begin{lem}\label{BL1}
Let $0<\lambda<d$ and let $f_n$ and $f$ be functions on $\R^d$ satisfying \eqref{punt} and \eqref{nonlocal}. Then
\begin{align}
\iint_{\R^d\times\R^d} \frac{(\overline{f_n(x)}-\overline{f(x)})f(x) \overline{f(y)} (f_n(y)-f(y))}{|x-y|^\lambda}dxdy 
= o(1) \label{con1} \,.
\end{align}
\end{lem}

\begin{proof}
By \eqref{punt} the functions $g_n(x,y) = \overline{f_n(x)} |x-y|^{-\lambda/2} f_n(y)$ converge almost everywhere in $\R^d\times\R^d$ to the function $g(x,y) = f(x) |x-y|^{-\lambda/2} f(y)$. Thus, by the Br\'ezis--Lieb lemma \cite{BL},
$$
\iint_{\R^d\times\R^d} |g_n(x,y)|^2 dxdy = \iint_{\R^d\times\R^d} |g(x,y)|^2 dxdy + \iint_{\R^d\times\R^d} |g_n(x,y)-g(x,y)|^2 dxdy + o(1) \,. 
$$
This is the same as
\begin{equation}
\label{eq:bl1}
\iint_{\R^d\times\R^d} \frac{\overline{f_n(x)} f(x) \overline{f(y)} f_n(y)}{|x-y|^\lambda}dxdy 
= \iint_{\R^d\times\R^d} \frac{|f(x)|^2 |f(y)|^2}{|x-y|^\lambda}dxdy + o(1) \,.
\end{equation}
On the other hand, applying the Br\'ezis--Lieb lemma to the functions $f_n(y)$ with respect to the measure $\int_{\R^d} |x-y|^{-\lambda} |f(x)|^2 dx dy$ we infer that
\begin{equation}
\label{eq:bl2}
\iint_{\R^d\times\R^d} \frac{|f(x)|^2 \left(\overline{f(y)} f_n(y)+ f(y)\overline{f_n(y)}\right)}{|x-y|^\lambda}dxdy 
= 2 \iint_{\R^d\times\R^d} \frac{|f(x)|^2 |f(y)|^2}{|x-y|^\lambda}dxdy + o(1) \,.
\end{equation}
The assertion of the lemma now follows from \eqref{eq:bl1} and \eqref{eq:bl2}.
\end{proof}

\begin{lem}\label{BL2}
Let $0<\lambda<d$, $2<p<\infty$ and let $f_n$ and $f$ be functions on $\R^d$ satisfying \eqref{punt}, \eqref{nonlocal} and \eqref{sobolev}. Then
\begin{align}
\iint_{\R^d\times\R^d} \frac{|f_n(x)|^2|f(y)|^2}{|x-y|^\lambda}dxdy 
= \iint_{\R^d\times\R^d} \frac{|f(x)|^2|f(y)|^2}{|x-y|^\lambda}dxdy + o(1) \label{con2} \,.
\end{align}
\end{lem}

\begin{proof}
Choose a subsequence $(f_{n_k})$ such that
\begin{equation}
\label{eq:BL2proof}
\lim_{k\to\infty} \iint_{\R^d\times\R^d} \frac{|f_{n_k}(x)|^2|f(y)|^2}{|x-y|^\lambda}dxdy
= \limsup_{n\to\infty} \iint_{\R^d\times\R^d} \frac{|f_n(x)|^2|f(y)|^2}{|x-y|^\lambda}dxdy \,.
\end{equation}
By a standard result in measure theory (see, e.g., \cite[Prop. 4.7.12]{Bo}), we infer from \eqref{punt} and \eqref{sobolev} that $|f_n|^2 \rightharpoonup |f|^2$ weakly in $L^{p/2}(\R^d)$.

Next, we note that the double integral in \eqref{nonlocal} coincides, up to a positive multiplicative constant, with the square of the $\dot H^{-(d-\lambda)/2}$-norm of $|f_n|^2$. Thus, by weak compactness we may assume (after passing to another subsequence, if necessary) that $|f_{n_k}|^2$ converges weakly in $\dot H^{-(d-\lambda)/2}$ to some $\sigma\in \dot H^{-(d-\lambda)/2}$, that is,
\begin{equation}
\label{eq:weakd}
\left( D^{-(d-\lambda)/2} \left( |f_{n_k}|^2 - \sigma \right), |D|^{-(d-\lambda)/2} \rho \right)_{L^2} \to 0
\end{equation}
for every $\rho\in\dot H^{-(d-\lambda)/2}$, where $D=\sqrt{-\Delta}$. Since $D^{-d+\lambda}$ is a bijection between $\dot H^{-(d-\lambda)/2}$ and $\dot H^{(d-\lambda)/2}$, we infer that
$$
\left( |f_{n_k}|^2 - \sigma, \psi \right)_{L^2} \to 0
$$
for every $\psi\in \dot H^{(d-\lambda)/2}$. Since $L^{p/2}\cap \dot H^{(d-\lambda)/2}$ is dense in $\dot H^{(d-\lambda)/2}$, we conclude that $\sigma=|f|^2$. Taking $\rho=|f|^2$ in \eqref{eq:weakd} we infer that
$$
\lim_{k\to\infty} \iint_{\R^d\times\R^d} \frac{|f_{n_k}(x)|^2|f(y)|^2}{|x-y|^\lambda}dxdy = \iint_{\R^d\times\R^d} \frac{|f(x)|^2|f(y)|^2}{|x-y|^\lambda}dxdy \,.
$$
This proves $\leq$ in \eqref{con2}. The proof of $\geq$ is similar, with $\limsup$ in \eqref{eq:BL2proof} replaced by $\liminf$.
\end{proof}

With Lemmas \ref{BL1} and \ref{BL2} at hand we are in position to give the

\begin{proof}[Proof of Lemma \ref{BL}]
We compute
\begin{align*}
\iint_{\R^d\times\R^d} \frac{|f_n(x)|^2|f_n(y)|^2}{|x-y|^\lambda}dxdy 
= \mathcal A & + \iint_{\R^d\times\R^d} \frac{|f_n(x)-f(x)|^2 |f_n(y)-f(y)|^2}{|x-y|^\lambda} dxdy \\
& - 4 \mathcal R_1 + 4 \mathcal R_2 -4 \mathcal R_3 \,,
\end{align*}
where, with $r_n(x)=\re \left(\overline{f(x)}\left(f_n(x)-f(x)\right)\right)$,
\begin{align*}
\mathcal A & = 2 \iint_{\R^d\times\R^d} \frac{|f(x)|^2|f_n(y)|^2}{|x-y|^\lambda}dxdy - \iint_{\R^d\times\R^d} \frac{|f(x)|^2|f(y)|^2}{|x-y|^\lambda}dxdy \,, \\
\mathcal R_1 & = \iint_{\R^d\times\R^d} \frac{r_n(x) r_n(y)}{|x-y|^\lambda}dxdy \,, \\
\mathcal R_2 & = \iint_{\R^d\times\R^d} \frac{r_n(x) |f_n(y)|^2}{|x-y|^\lambda}dxdy  \,, \\
\mathcal R_3 & = \iint_{\R^d\times\R^d} \frac{r_n(x) |f(y)|^2}{|x-y|^\lambda}dxdy
\end{align*}
We claim that Lemma \ref{BL1} implies that
\begin{equation}
\label{eq:r1}
\mathcal R_1 = o(1) \,.
\end{equation}
To see this, let $s_n(x)=\overline{f(x)}\left(f_n(x)-f(x)\right)$. Then$r_n=(s_n+\overline{s_n})/2$ and
$$
\mathcal R_1 = \frac12 \re \iint_{\R^d\times\R^d} \frac{\overline{s_n(x)} s_n(y)}{|x-y|^\lambda}dxdy + \frac12 \re \iint_{\R^d\times\R^d} \frac{s_n(x) s_n(y)}{|x-y|^\lambda}dxdy \,.
$$
Lemma \ref{BL1} says that
$$
\iint_{\R^d\times\R^d} \frac{\overline{s_n(x)} s_n(y)}{|x-y|^\lambda}dxdy = o(1) \,.
$$
Thus, to prove \eqref{eq:r1} is suffices to prove that the double integral involving the product $s_n(x) s_n(y)$ tends to zero. Since the kernel $|x-y|^{-\lambda}$ is non-negative definite (i.e., has a non-negative Fourier transform), we have the Cauchy--Schwarz inequality
$$
\left| \iint_{\R^d\times\R^d} \frac{\overline{g(x)}h(y)}{|x-y|^\lambda} dxdy \right|
\leq \left( \iint_{\R^d\times\R^d} \frac{\overline{g(x)}g(y)}{|x-y|^\lambda} dxdy \right)^{1/2}
\left(\iint_{\R^d\times\R^d} \frac{\overline{h(x)}h(y)}{|x-y|^\lambda} dxdy
\right)^{1/2}
$$
for any $g$ and $h$ for which the right side is finite. Applying this to $g=\overline{s_n}$ and $h=s_n$ and using again Lemma \ref{BL1}, we obtain \eqref{eq:r1}.

A similar application of the Cauchy--Schwarz inequality implies
$$
|\mathcal R_2| \leq \left( \iint_{\R^d\times\R^d} \frac{|f_n(x)|^2 |f_n(y)|^2}{|x-y|^\lambda}dxdy \right)^{1/2} \mathcal R_1^{1/2}
$$
and
$$
|\mathcal R_3| \leq \left( \iint_{\R^d\times\R^d} \frac{|f(x)|^2 |f(y)|^2}{|x-y|^\lambda}dxdy \right)^{1/2} \mathcal R_1^{1/2} \,.
$$
Therefore, by \eqref{nonlocal} and again by Lemma \ref{BL1}, $\mathcal R_2 + \mathcal R_3 = o(1)$.

Finally, we use assumption \eqref{sobolev} to apply Lemma \ref{BL2} and to conclude that
$$
\mathcal A = \iint_{\R^d\times\R^d} \frac{|f(x)|^2|f(y)|^2}{|x-y|^\lambda}dxdy + o(1) \,.
$$
This completes the proof of Lemma \ref{BL}.
\end{proof}


\appendix

\section{Proof of Proposition \ref{CHLSgen}}\label{section2}

In this appendix we extend the method from \cite{BOV} to derive Proposition \ref{CHLSgen}. Let $d\geq 1$, $s>0$ and $0<\lambda<d$ be fixed. First note that if $d>2s$ and $p=d/(d-2s)$, then the claimed inequality is the Sobolev inequality.

Therefore, in the following we assume that $p\neq d/(d-2s)$. Note that this implies, in particular, that $\lambda\neq 4s$. A tedious but straightforward computation shows that the restriction on $p$ in the proposition (in the case $p\neq d/(d=2s)$) is equivalent to the fact that 
\begin{equation}
\label{eq:theta}
\theta = \frac{2d-2pd+p\lambda}{d-2ps-pd+p\lambda} \,.
\end{equation}
satisfies the bounds
\begin{equation}
\label{eq:thetabounds}
\frac{d-\lambda}{d+2s -\lambda}\leq \theta < 1 \,.
\end{equation}
Note that $\theta$ is always well-defined in $\R\cup\{\pm\infty\}$, since numerator and denominator do not vanish simultaneously (since $\lambda\neq 4s$).

We now recall a version of the fractional Gagliardo--Nirenberg inequality.

\begin{lem}\label{fracGN}
Let $p\in(1,\infty)$ and define $\theta$ by \eqref{eq:theta}. If $\theta$ satisfies \eqref{eq:thetabounds}, then
\begin{equation}\label{intermGN}
\|D^{\frac{d-\lambda}{2}} \psi\|_{L^p}\leq C \|\psi\|_{L^2}^{1-\theta} 
\|D^{s+\frac{d-\lambda}{2}} \psi\|_{L^\frac{2p}{p+1}}^\theta \,.
\end{equation}
If $d<2s$, then \eqref{intermGN} remains valid for $p=\infty$. 
\end{lem}

For $p<\infty$ this is explicitly stated in \cite{NP}, but we use the opportunity to show how it can be deduced from the older results of \cite{GiVe}. As similar argument works for inequality \eqref{GagliNiren}. An advantage of our reduction to \cite{GiVe} is that it clearly shows why $p=(d+4s-\lambda)/(d+2s-\lambda)$ is an endpoint.

\begin{proof}
The case $p<\infty$ can be easily deduced from \cite{GiVe}. More precisely, for \emph{any} $1<p<\infty$ and \emph{any} $0\leq\theta\leq 1$, we have the convexity inequality
$$
\|\psi\|_{L^2}^{1-\theta} 
\|D^{s+\frac{d-\lambda}{2}} \psi\|_{L^\frac{2p}{p+1}}^\theta
\geq c \|D^{\theta(1+\frac{d-\lambda}{2})}\psi\|_{L^{\frac{2p}{p+\theta}}}
$$
with some $c>0$ \cite[Lemma A.1]{GiVe}. If $p=(d+4s-\lambda)/(d+2s-\lambda)$, then $\theta$ defined in \eqref{eq:theta} equals $(d-\lambda)/(d+2s-\lambda)$ and the lemma follows. Otherwise, for the choice \eqref{eq:theta} satisfying \eqref{eq:thetabounds} we can apply \cite[Lemma A.2]{GiVe} to deduce that
$$
\|D^{\theta(1+\frac{d-\lambda}{2})}\psi\|_{L^{\frac{2p}{p+\theta}}}
\geq c' \|D^{\frac{d-\lambda}{2}} \psi\|_{L^p} \,,
$$
which completes the proof of the lemma for $p<\infty$.

If $p=\infty$ and $d<2s$, then by the Cauchy--Schwarz inequality
\begin{align*}
\left|D^{\frac{d-\lambda}{2}}\psi(x)\right| & = (2\pi)^{-d/2}\left| \int_{\R^d} |\xi|^{\frac{d-\lambda}2} e^{i\xi\cdot x} \hat\psi(\xi) \,dx \right| \\
& \leq (2\pi)^{-d/2} \int_{|\xi|<R} |\xi|^{\frac{d-\lambda}2} |\hat\psi(\xi)| \,d\xi + (2\pi)^{-d/2} \int_{|\xi|\geq R} |\xi|^{\frac{d-\lambda}2} |\hat\psi(\xi)| \,d\xi \\
& \leq C_1 R^{\frac{2d-\lambda}2} \left( \int_{|\xi|<R} |\hat\psi(\xi)|^2 \,d\xi \right)^{1/2}
+C_2 R^{-\frac{2s-d}2} \left( \int_{|\xi|<R} |\xi|^{2s+d-\lambda} |\hat\psi(\xi)|^2 \,d\xi \right)^{1/2} \\
& \leq C_1 R^{\frac{2d-\lambda}2} \|\psi\|_2 + C_2 R^{-\frac{2s-d}2} \|D^{s+\frac{d-\lambda}{2}} \psi\|_{L^2} \,.
\end{align*}
Optimizing in $R$ yields the claimed inequality.
\end{proof}

We combine Lemma \ref{fracGN} with the identity
$$
\iint_{\R^d\times \R^d}
\frac{|\varphi(x)|^2 |\varphi(y)|^2}{|x-y|^{\lambda}}dxdy=
C_{d,\lambda} \left\| D^{-\frac{d-\lambda}{2}} |\varphi|^2 \right\|_{L^2}^2
$$
for some positive constant $C_{d,\lambda}$, and obtain
\begin{align*}
\left\|\varphi\right\|_{L^{2p}}^2 = \left\| |\varphi|^2 \right\|_{L^p} & \leq C \left\| D^{-\frac{d-\lambda}{2}}|\varphi|^2 \right\|_{L^2}^{1-\theta} 
\left\| D^{s} |\varphi|^2 \right\|_{L^\frac{2p}{p+1}}^\theta \\
& = C C_{d,\lambda}^{-\frac{1-\theta}2} \left( \iint_{\R^d\times \R^d}
\frac{|\varphi(x)|^2 |\varphi(y)|^2}{|x-y|^{\lambda}}dxdy \right)^{\frac{1-\theta}2}
\left\| D^{s} |\varphi|^2 \right\|_{L^\frac{2p}{p+1}}^\theta \,.
\end{align*}
By the fractional chain-rule (see, e.g., \cite{CW}), this implies
$$
\|\varphi\|_{L^{2p}}^2\leq C' \left(\iint_{\R^d\times \R^d}
\frac{|\varphi(x)|^2 |\varphi(y)|^2}{|x-y|^{\lambda}}dxdy \right)^{\frac{1-\theta}2}
\left\|D^s \varphi\right\|_{L^2}^\theta \|\varphi\|_{L^{2p}}^{\theta} \,.
$$
(Strictly speaking, \cite{CW} contains only the case $0<s<1$ and $p<\infty$. A proof with $(1+D^2)^{1/2}$ instead of $D$ is contained in \cite[Prop. 1.1]{Ta}; 
see also \cite{GO} and references therein for the general case.) To summarize, we have shown that
$$
\|\varphi\|_{L^{2p}} \leq \left(C'\right)^{\frac{1}{2-\theta}} \left( \iint_{\R^d\times \R^d}
\frac{|\varphi(x)|^2 |\varphi(y)|^2}{|x-y|^{\lambda}}dxdy\right)^{\frac{1-\theta}{4-2\theta}}
\|\varphi\|_{\dot H^s}^\frac{\theta}{2-\theta} \,,
$$
which is the stated inequality.\qed

\subsection*{Acknowledgement.} FIRB2012 `Dinamiche dispersive: analisi di
Fourier e metodi variazionali' (J.B., N.V.) and PRIN2009 `Metodi Variazionali e Topologici nello Studio di Fenomeni non Lineari' (J.B.) and U.S. National Science Foundation grant PHY-1068285 (R.F.) are acknowledged. The authors would like to thank E. Lieb and G. Ponce for useful discussions.


\end{document}